\documentclass[twoside,english,headsepline]{scrartcl}
\pdfoutput=1
\usepackage[T1]{fontenc}
\usepackage[latin9]{inputenc}
\usepackage[a4paper]{geometry}
\usepackage{babel}
\usepackage{refstyle}
\usepackage{mathtools}
\usepackage{amsmath}
\usepackage{amsthm}
\usepackage{amssymb}
\usepackage{esint}
\usepackage[authoryear]{natbib}
\usepackage{lmodern}
\usepackage{stix}
\usepackage[kerning=true]{microtype}
\usepackage[cmyk]{xcolor}
\usepackage[unicode=true,pdfusetitle,
 bookmarks=true,bookmarksnumbered=false,bookmarksopen=false,
 breaklinks=false,pdfborder={0 0 0},backref=false,colorlinks=true,
 urlcolor=blue]{hyperref}

\usepackage{doi}

\hypersetup{pdftitle={Existence and uniqueness of steady weak solutions to the Navier--Stokes equations in R\texttwosuperior},
 pdfauthor={Julien Guillod \& Peter Wittwer}}
\usepackage{scrlayer-scrpage}
\ohead{\pagemark}
\rehead{J. Guillod \& P. Wittwer}
\lohead{Existence and uniqueness of steady weak solutions to the Navier--Stokes equations in $\mathbb{R}^{2}$}
\ofoot{}
\theoremstyle{plain}
\newtheorem{thm}{Theorem}
\newtheorem{prop}[thm]{Proposition}
\newtheorem{lem}[thm]{Lemma}

\theoremstyle{definition}
\newtheorem{defn}[thm]{Definition}
\theoremstyle{remark}
\newtheorem{rem}[thm]{Remark}

\SetExtraKerning[unit=space]
    {encoding={*}, family={stix}, series={*}, size={footnotesize,small,normalsize}}
    {\textendash={400,400}}

\makeatletter
\let\amstexbig\big
\def\newbig#1{%
  \ifx#1|%
    \expandafter\@firstoftwo
  \else
    \expandafter\@secondoftwo
  \fi
  {\big@bar}%
  {\amstexbig{#1}}%
}
\AtBeginDocument{\let\big\newbig}
\def\big@bar{\bBigg@{1.1}|}
\makeatother

\makeatletter
\let\save@mathaccent\mathaccent
\newcommand*\if@single[3]{%
  \setbox0\hbox{${\mathaccent"0362{#1}}^H$}%
  \setbox2\hbox{${\mathaccent"0362{\kern0pt#1}}^H$}%
  \ifdim\ht0=\ht2 #3\else #2\fi
  }
\newcommand*\rel@kern[1]{\kern#1\dimexpr\macc@kerna}
\newcommand*\mybar[1]{\@ifnextchar^{{\wide@bar{#1}{0}}}{\wide@bar{#1}{1}}}
\newcommand*\wide@bar[2]{\if@single{#1}{\wide@bar@{#1}{#2}{1}}{\wide@bar@{#1}{#2}{2}}}
\newcommand*\wide@bar@[3]{%
  \begingroup
  \def\mathaccent##1##2{%
    \let\mathaccent\save@mathaccent
    \if#32 \let\macc@nucleus\first@char \fi
    \setbox\z@\hbox{$\macc@style{\macc@nucleus}_{}$}%
    \setbox\tw@\hbox{$\macc@style{\macc@nucleus}{}_{}$}%
    \dimen@\wd\tw@
    \advance\dimen@-\wd\z@
    \divide\dimen@ 3
    \@tempdima\wd\tw@
    \advance\@tempdima-\scriptspace
    \divide\@tempdima 10
    \advance\dimen@-\@tempdima
    \ifdim\dimen@>\z@ \dimen@0pt\fi
    \rel@kern{0.6}\kern-\dimen@
    \if#31
      \overbracket[0.65pt][0pt]{\rel@kern{-0.6}\kern\dimen@\macc@nucleus\rel@kern{0.4}\kern\dimen@}%
      \advance\dimen@0.1\dimexpr\macc@kerna
      \let\final@kern#2%
      \ifdim\dimen@<\z@ \let\final@kern1\fi
      \if\final@kern1 \kern-\dimen@\fi
    \else
       \overbracket[0.65pt][0pt]{\rel@kern{-0.6}\kern\dimen@#1}%
    \fi
  }%
  \macc@depth\@ne
  \let\math@bgroup\@empty \let\math@egroup\macc@set@skewchar
  \mathsurround\z@ \frozen@everymath{\mathgroup\macc@group\relax}%
  \macc@set@skewchar\relax
  \let\mathaccentV\macc@nested@a
  \if#31
    \macc@nested@a\relax111{#1}%
  \else
    \def\gobble@till@marker##1\endmarker{}%
    \futurelet\first@char\gobble@till@marker#1\endmarker
    \ifcat\noexpand\first@char A\else
      \def\first@char{}%
    \fi
    \macc@nested@a\relax111{\first@char}%
  \fi
  \endgroup
}
\makeatother

\global\long\def\rd{\mathrm{d}}
\global\long\def\bu{\boldsymbol{u}}
\global\long\def\bv{\boldsymbol{v}}
\global\long\def\bw{\boldsymbol{w}}

\global\long\def\bd{\boldsymbol{d}}

\global\long\def\bx{\boldsymbol{x}}

\global\long\def\be{\boldsymbol{e}}
\global\long\def\bff{\boldsymbol{f}}

\global\long\def\bF{\boldsymbol{F}}
\global\long\def\bFn{\boldsymbol{R}_{n}}
\global\long\def\bbF{\mathbf{F}}

\global\long\def\bn{\boldsymbol{n}}
\global\long\def\ba{\boldsymbol{a}}
\global\long\def\bc{\boldsymbol{c}}

\global\long\def\bphi{\boldsymbol{\varphi}}
\global\long\def\bpsi{\boldsymbol{\psi}}

\global\long\def\blambda{\boldsymbol{\mu}}
\global\long\def\bnabla{\boldsymbol{\nabla}}
\global\long\def\bcdot{\boldsymbol{\cdot}}

\global\long\def\bzero{\boldsymbol{0}}
\global\long\def\widebar#1{\mybar{#1}}
\global\long\def\loc{\mathrm{loc}}
\global\long\def\weight{\mathfrak{w}}
\global\long\def\cAn{\boldsymbol{\mathcal{A}}_{n}}
\global\long\def\cBn{\boldsymbol{\mathcal{B}}_{n}}
\global\long\def\cLn{\boldsymbol{\mathcal{L}}_{n}}

\begin{document}

\title{Existence and uniqueness\\
of steady weak solutions to the\\
Navier\textendash Stokes equations in $\mathbb{R}^{2}$}

\author{Julien Guillod\textsuperscript{1,2} \qquad{} Peter Wittwer\textsuperscript{3}\\
{\small{}\textsuperscript{1}Mathematics Department, Princeton University}\\
{\small{}\textsuperscript{2}ICERM, Brown University}\\
{\small{}\textsuperscript{3}Department of Theoretical Physics, University
of Geneva}}

\date{March 20, 2017}

\maketitle

\begin{abstract}
The existence of weak solutions to the stationary Navier\textendash Stokes
equations in the whole plane $\mathbb{R}^{2}$ is proven. This particular
geometry was the only case left open since the work of Leray in 1933.
The reason is that due to the absence of boundaries the local behavior
of the solutions cannot be controlled by the enstrophy in two dimensions.
We overcome this difficulty by constructing approximate weak solutions
having a prescribed mean velocity on some given bounded set. As a
corollary, we obtain infinitely many weak solutions in $\mathbb{R}^{2}$
parameterized by this mean velocity, which is reminiscent of the expected
convergence of the velocity field at large distances to any prescribed
constant vector field. This explicit parameterization of the weak
solutions allows us to prove a weak-strong uniqueness theorem for
small data. The question of the asymptotic behavior of the weak solutions
remains however open, when the uniqueness theorem doesn't apply.
\end{abstract}
\textsf{\textbf{Keywords}}\quad{}Navier\textendash Stokes equations,
Steady weak solutions, Whole plane\\
\textsf{\textbf{MSC classes}}\quad{}76D03, 76D05, 35D30, 35A01, 35J60

\section{Introduction}

We consider the stationary Navier\textendash Stokes equations in an
exterior domain $\Omega=\mathbb{R}^{n}\setminus\widebar B$ where
$B$ is a bounded simply connected Lipschitz domain,
\begin{align}
\Delta\bu-\bnabla p & =\bu\bcdot\bnabla\bu+\bff\,, & \bnabla\bcdot\bu & =0\,, & \left.\bu\right|_{\partial\Omega} & =\bu^{*}\,,\label{eq:intro-ns-eq}
\end{align}
with a given forcing term $\bff$ and a boundary condition $\bu^{*}$
if $B$ is not empty. Since the domain is unbounded, we add the following
boundary condition at infinity,
\begin{equation}
\lim_{\left|\bx\right|\to\infty}\bu(\bx)=\bu_{\infty}\,,\label{eq:intro-ns-limit}
\end{equation}
where $\bu_{\infty}\in\mathbb{R}^{n}$ is a constant vector. In his
seminal work, \citet{Leray-Etudedediverses1933} proposed a three-step
method to show the existence of weak solutions to this problem. First,
the boundary conditions $\bu^{*}$ and $\bu_{\infty}$ are lifted
by an extension $\ba$ which satisfies the so-called extension condition.
The second step is to show the existence of weak solutions in bounded
domains. Finally, the third step is to define a sequence of invading
bounded domains that coincide in the limit with the unbounded domain
and show that the induced sequence of solutions converges in some
suitable space. With this strategy, \citet{Leray-Etudedediverses1933}
was able to construct weak solutions in domains with a compact boundary
if the flux through each connected component of the boundary is zero.
The extension of this result to the case where the fluxes are small
was done by \citet[Section X.4]{Galdi-IntroductiontoMathematical2011}
in three dimensions and by \citet{Russo-NoteExteriorTwo-Dimensional2009}
in two dimensions. We note that by elliptic regularity, weak solutions
are automatically two derivatives more regular than the data \citep[Theorem X.1.1]{Galdi-IntroductiontoMathematical2011}.
All these results about weak solutions have essentially only two drawbacks,
both in two dimensions: the validity of \eqref{intro-ns-limit} is
not known and the method of Leray cannot be applied if $\Omega=\mathbb{R}^{2}$.

In three dimensions, the method of Leray can be used to prove the
existence of a weak solution satisfying \eqref{intro-ns-limit} for
any $\bu_{\infty}\in\mathbb{R}^{3}$. By assuming the existence of
a strong solution satisfying various decay conditions at infinity,
\citet{Kozono-unbounded1993} and \citet[\S X.3]{Galdi-IntroductiontoMathematical2011}
proved the uniqueness of weak solutions satisfying the energy inequality.
Moreover, the asymptotic behavior was determined by \citet[Theorem X.8.1]{Galdi-IntroductiontoMathematical2011}
if $\bu_{\infty}\neq\bzero$ and by \citet[Theorem 1]{Korolev.Sverak-largedistanceasymptotics2011}
if $\bu_{\infty}=\bzero$ and the data are small enough. Therefore,
in three dimensions the picture is pretty complete.

In two dimensional exterior domains, the homogeneous Sobolev space
$\dot{H}^{1}(\Omega)$ used in the construction of weak solutions
is too weak to determine the validity of \eqref{intro-ns-limit},
because elements in this function space can even grow at infinity.
Therefore, the results concerning the uniqueness and the asymptotic
behavior of weak solutions in two dimensions are very limited. Concerning
the asymptotic behavior, \citet{Gilbarg.Weinberger-AsymptoticPropertiesof1974,Gilbarg.Weinberger-Asymptoticpropertiesof1978}
proved that either there exists $\bu_{0}\in\mathbb{R}^{2}$ such that
\[
\lim_{\left|\bx\right|\to\infty}\int_{S^{1}}\left|\bu-\bu_{0}\right|^{2}=0\,,\qquad\text{or}\qquad\lim_{\left|\bx\right|\to\infty}\int_{S^{1}}\left|\bu\right|^{2}=\infty\,.
\]
Later on \citet{Amick-Leraysproblemsteady1988} showed that if $\bu^{*}=\bff=\bzero$,
then $\bu\in L^{\infty}(\Omega)$ so that the first alternative must
apply for some $\bu_{0}$. Nevertheless, the question if any prescribed
value at infinity $\bu_{\infty}$ can be obtained this way remains
open in general. For small data and $\bu_{\infty}\neq\bzero$, \citet{Finn-stationarysolutionsNavier1967}
constructed strong solutions satisfying \eqref{intro-ns-limit}. By
assuming that the domain is centrally symmetric, \citet[Theorem 2.27]{Guillod-review2015}
proved the existence of a weak solution with $\bu_{\infty}=\bzero$.
Under additional symmetry assumptions, the existence and asymptotic
decay of solutions with $\bu_{\infty}=\bzero$ was proven under suitable
smallness assumptions \citep{Yamazaki-stationaryNavier-Stokesequation2009,Yamazaki-Uniqueexistence2011,Pileckas-existencevanishing2012,Guillod-review2015}
or specific boundary conditions \citep{Hillairet-mu2013}. We refer
the reader to \citet[Chapter XII]{Galdi-IntroductiontoMathematical2011}
and \citet{Guillod-review2015} for a more complete discussion on
the asymptotic behavior of solutions in two-dimensional unbounded
domains. The question of the uniqueness of weak solutions for small
data is even more open in two-dimensional exterior domains. The reason
is that the value at infinity $\bu_{\infty}$ should be intuitively
part of the data in order to expect uniqueness. The only known results
in that direction are due to \citet{Yamazaki-Uniqueexistence2011}
and \citet{Nakatsuka-uniquenessofsymmetric2013}, who proved the uniqueness
of weak solutions satisfying the energy inequality under suitable
symmetry and smallness assumptions.

The other main issue concerns the construction of weak solutions in
$\Omega=\mathbb{R}^{2}$, which fails due to a fundamental issue with
the function space \citep[Remark X.4.4 \& Section XII.1]{Galdi-IntroductiontoMathematical2011}.
More precisely the completion $\dot{H}_{0}^{1}(\Omega)$ of smooth
compactly supported functions in the semi-norm of $\dot{H}^{1}(\Omega)$
can be viewed as a space of locally defined functions only if $\Omega\neq\mathbb{R}^{2}$.
The example of \citet[Remarque 4.1]{Deny-Lesespaces1954} shows that
the elements of $\dot{H}_{0}^{1}(\mathbb{R}^{2})$ are equivalence
classes and cannot be viewed as functions. The reason is that constant
functions can be approximated by compactly supported functions in
$\dot{H}^{1}(\mathbb{R}^{2})$, hence the function cannot be locally
bounded by its gradient. This can also be viewed as a consequence
of the absence of Poincaré inequality in $\dot{H}^{1}(\mathbb{R}^{2})$.

The main result of this paper (\thmref{weak-solution}) is a modification
of the method of Leray which allows to construct weak solutions in
$\Omega=\mathbb{R}^{2}$. The idea is to construct approximate solutions
in invading balls having a prescribed mean on some fixed bounded set.
This can be done by using the freedom in the choice of the boundary
condition on the boundary of the balls. That way, the local properties
of the approximate solutions are controlled and can be used to prove
that the sequence of approximate solutions converges locally in $L^{p}$-spaces.
The method we are using furnishes as a corollary infinitely many weak
solutions parameterized by the mean $\blambda=\fint_{\omega}\bu$,
where $\omega$ is a fixed bounded set of positive measure. Intuitively
we have recovered the parameter $\bu_{\infty}\in\mathbb{R}^{2}$,
even if the validity of \eqref{intro-ns-limit} remains open. However,
the explicit parametrization by $\blambda$, can be used to prove
a weak-strong uniqueness theorem for small solutions (\thmref{uniqueness}).
This is done in the spirit of what is known in three dimensions \citep[Theorem X.3.2]{Galdi-IntroductiontoMathematical2011}
and is the first general uniqueness result available in two dimensions.
We remark that the existence of a parametrization of the two-dimensional
weak solutions by two real parameters is open when $\partial\Omega\neq\emptyset$,
and in this case it is not clear that the mean $\blambda=\fint_{\omega}\bu$
will be such a parametrization. A more detailed discussion of the
results is added at the end of \secref{main}.

\paragraph{Notations}

The open ball of radius $n$ centered at the origin is denoted by
$B_{n}$. For $\bx\in\mathbb{R}^{d}$, we define $\langle\bx\rangle=1+\left|\bx\right|$
and the weight $\weight(\bx)=\bigl[\langle\bx\rangle\langle\log\langle\bx\rangle\rangle\bigr]^{-1}$.
The mean value of a vector field on a bounded set $\omega$ of positive
measure is written as $\fint_{\omega}\bu=\frac{1}{\left|\omega\right|}\int_{\omega}\bu$.
The space of smooth solenoidal functions having compact support in
$\Omega$ is denoted by $C_{0,\sigma}^{\infty}(\Omega)$. We denote
by $\dot{H}^{1}(\Omega)$ the linear space $\left\{ \bu\in L_{\loc}^{1}(\Omega):\,\bnabla\bu\in L^{2}(\Omega)\right\} $
with the semi-norm $\left\Vert \bu\right\Vert _{\dot{H}^{1}(\Omega)}=\left\Vert \bnabla\bu\right\Vert _{L^{2}(\Omega)}$.
The subspace of weakly divergence-free vectors fields in $\dot{H}^{1}(\Omega)$
is written as $\dot{H}_{\sigma}^{1}(\Omega)$. Let $\dot{H}_{0,\sigma}^{1}(\Omega)$
denote the completion of $C_{0,\sigma}^{\infty}(\Omega)$ in the semi-norm
of $\dot{H}^{1}(\Omega)$.

\section{\label{sec:main}Main results}

We first recall the standard notion of weak solutions to the stationary
Navier\textendash Stokes equations:
\begin{defn}
\label{def:weak-solution}Let $\Omega\subset\mathbb{R}^{2}$ be any
Lipshitz domain (in particular $\Omega=\mathbb{R}^{2}$ is allowed).
Given $\bu^{*}\in W^{1/2,2}(\partial\Omega)$ and a rank-two tensor
$\bbF\in L^{2}(\Omega)$, a vector field $\bu:\Omega\to\mathbb{R}^{n}$
is called a weak solution of the Navier\textendash Stokes equations
\eqref{intro-ns-eq} in $\Omega$ with $\bff=\bnabla\bcdot\bbF$ if
\begin{enumerate}
\item $\bu\in\dot{H}_{\sigma}^{1}(\Omega)$\,;
\item $\left.\bu\right|_{\partial\Omega}=\bu^{*}$ in the trace sense\,;
\item $\bu$ satisfies
\begin{equation}
\bigl\langle\bnabla\bu,\bnabla\bphi\bigr\rangle_{L^{2}(\Omega)}+\bigl\langle\bu\bcdot\bnabla\bu,\bphi\bigr\rangle_{L^{2}(\Omega)}=\bigl\langle\bbF,\bnabla\bphi\bigr\rangle_{L^{2}(\Omega)}\label{eq:weak-solution}
\end{equation}
for all $\bphi\in C_{0,\sigma}^{\infty}(\Omega)$\,.
\end{enumerate}
\end{defn}
The existence of weak solutions in two-dimensional unbounded domains
was first proved by \citet{Leray-Etudedediverses1933} for vanishing
flux through the boundaries and was extended to the case of small
fluxes by \citet{Russo-NoteExteriorTwo-Dimensional2009}:
\begin{thm}
\label{thm:weak-exterior}Let $\Omega\subset\mathbb{R}^{2}$ be an
exterior domain having a compact connected Lipschitz boundary $\partial\Omega\neq\emptyset$.
Let $\bu^{*}\in W^{1/2,2}(\partial\Omega)$ and $\bbF\in L^{2}(\Omega)$.
If the flux
\[
\Phi=\int_{\partial\Omega}\bu^{*}\bcdot\bn\,,
\]
satisfies $\left|\Phi\right|<2\pi$, then there exists a weak solution
$\bu\in\dot{H}_{\sigma}^{1}(\Omega)$ of the Navier\textendash Stokes
equations \eqref{intro-ns-eq} in $\Omega$.\end{thm}
\begin{rem}
\label{rem:exterior-L2-compact}For $\partial\Omega\neq\emptyset$,
if $\bff\in L^{2}(\Omega)$ is a source term of compact support, then
there exists $\bbF\in L^{2}(\Omega)$ such that $\bff=\bnabla\bcdot\bF$.
See \lemref{exterior-representation} for a more general result in
this direction.
\end{rem}

\begin{rem}
This result can be easily extended to the case where the boundary
$\partial\Omega$ has finitely many connected components, provided
the flux through each connected component is small enough.
\end{rem}

\begin{rem}
The three-dimensional analogue of this theorem is valid even if $\partial\Omega=\emptyset$,
\emph{i.e.} if $\Omega=\mathbb{R}^{3}$, see \citet[Theorem X.4.1]{Galdi-IntroductiontoMathematical2011}.
\end{rem}
As explained in the introduction, the method used to prove \thmref{weak-exterior}
fails for $\Omega=\mathbb{R}^{2}$. Our main result is the existence
of infinitely many weak solutions in $\mathbb{R}^{2}$ for every given
$\bF$:
\begin{thm}
\label{thm:weak-solution}Let $\Omega=\mathbb{R}^{2}$ and $\omega\subset\Omega$
be a bounded subset of positive measure. Let $\bbF\in L^{2}(\Omega)$
be a rank-two tensor. Then for any $\blambda\in\mathbb{R}^{2}$, there
exists a weak solution $\bu\in\dot{H}_{\sigma}^{1}(\Omega)$ of the
Navier\textendash Stokes equations \eqref{intro-ns-eq} in $\Omega$
such that $\fint_{\omega}\bu=\blambda$. Moreover,
\begin{equation}
\bigl\Vert\bnabla\bu\bigr\Vert_{L^{2}(\Omega)}^{2}\leq\bigl\langle\bbF,\bnabla\bu\bigr\rangle_{L^{2}(\Omega)}\,,\label{eq:energy-inequality}
\end{equation}
so $\left\Vert \bnabla\bu\right\Vert _{L^{2}(\Omega)}\leq\left\Vert \bbF\right\Vert _{L^{2}(\Omega)}$.\end{thm}
\begin{rem}
\label{rem:R2-L2-compact}For $\Omega=\mathbb{R}^{2}$, if $\bff\in L^{2}(\Omega)$
is a source term of compact support and $\int_{\Omega}\bff=\bzero$,
then there exists $\bbF\in L^{2}(\Omega)$ such that $\bff=\bnabla\bcdot\bF$.
See \lemref{R2-representation} for a more general result in this
direction.
\end{rem}

\begin{rem}
In this result the set $\omega$ can be easily replaced by a bounded
and uniformly Lipschitz arc $\omega\subset\mathbb{R}^{2}$ of positive
one-dimensional measure.
\end{rem}
Finally, with our parametrization of weak solutions by the average
$\blambda$, we can prove a weak-strong uniqueness theorem for small
data:
\begin{thm}
\label{thm:uniqueness}Let $\Omega=\mathbb{R}^{2}$ and $\omega\subset\Omega$
be a bounded subset of positive measure. Let $\bu$ and $\tilde{\bu}$
be two weak solutions of the Navier\textendash Stokes equations \eqref{intro-ns-eq}
in $\Omega$ for the same source term $\bF\in L^{2}(\Omega)$, having
the same mean value $\fint_{\omega}\bu=\fint_{\omega}\tilde{\bu}$,
and satisfying the energy inequality \eqref{energy-inequality}. There
exists $\delta>0$ depending only on $\omega$ such that if
\begin{equation}
\left|\tilde{\bu}(\bx)-\bu_{\infty}\right|\leq\frac{\delta}{\langle\bx\rangle\langle\log\langle\bx\rangle\rangle}\,,\label{eq:bound-ubar}
\end{equation}
for some $\bu_{\infty}\in\mathbb{R}^{2}$, then $\bu=\tilde{\bu}$.
\end{thm}
We now discuss our results in more detail. The space $\dot{H}^{1}(\Omega)$
is not a Banach space since the constant vector fields are in the
kernel of the semi-norm, but $\dot{H}^{1}(\Omega)$ can be viewed
as a sort of graded space. In the presence of a nontrivial boundary,
this problem can be fixed by using the completion $\dot{H}_{0}^{1}(\Omega)$
of smooth compactly supported functions in the semi-norm of $\dot{H}^{1}(\Omega)$.
Intuitively, there is no more freedom in the choice of the constant,
since the elements of $\dot{H}_{0}^{1}(\Omega)$ are vanishing on
the boundary $\partial\Omega$.

When the boundary is trivial, \emph{i.e.} $\Omega=\mathbb{R}^{n}$,
the boundary can not serve as an anchor anymore to fix the problem
of the constants. The solution of this problem now depends on the
dimension. For $\Omega=\mathbb{R}^{3}$, the constants do not belong
to the completion $\dot{H}_{0}^{1}(\Omega)$, the reason being the
Sobolev embedding into $L^{6}(\Omega)$. Therefore, the space $\dot{H}^{1}(\Omega)$
is in some sense naturally graded by the constant at infinity $\bu_{\infty}\in\mathbb{R}^{3}$
in three dimensions.

For $\Omega=\mathbb{R}^{2}$, the constants belong to the completion
$\dot{H}_{0}^{1}(\Omega)$ of smooth compactly supported functions
in the semi-norm of $\dot{H}^{1}(\Omega)$, so $\dot{H}_{0}^{1}(\Omega)$
is a space of equivalence classes defined by the relation of being
equal up to a constant vector field. Therefore, $\dot{H}_{0}^{1}(\Omega)$
cannot be viewed as a space of locally defined functions. To overcome
this difficulty, we choose to graduate the space $\dot{H}^{1}(\Omega)$
by the mean $\blambda\in\mathbb{R}^{2}$ of the vector field on $\omega$.
Intuitively, this is a recovery of the parameter $\bu_{\infty}\in\mathbb{R}^{2}$,
which is lost in two dimensions during the completion. This new way
of parameterizing the function space in two dimensions is crucial
to prove the existence of weak solutions and also for the weak-strong
uniqueness result.

Concerning our weak-strong uniqueness result, we note that we don't
except the existence of a solution $\tilde{\bu}$ satisfying \eqref{bound-ubar}
for all $\bF\in L^{2}(\Omega)$. In fact, we can easily construct
counterexamples. For $\bu_{\infty}\neq\bzero$, the derivative of
a suitable smoothing of the Oseen fundamental solution will typically
decay at infinity like $|\bx|^{-1}$ in the wake and will be a weak
solution for a particular forcing. For $\bu_{\infty}=\bzero$, the
smoothing of the exact solution $\bx^{\perp}|\bx|^{-2}$ will also
be an exact solution decaying like $|\bx|^{-1}$ for a forcing term
of compact support. However, by using the asymptotic behavior proven
by \citet[Theorem 6.1]{Babenko-AsymptoticBehaviorof1970}, we can
deduce some compatibility conditions on $\bff$ such that the existence
of a solution $\tilde{\bu}$ satisfying \eqref{bound-ubar} with $\bu_{\infty}\neq\bzero$
can be deduced. For $\bu_{\infty}=\bzero$, it was conjectured that
some solutions could even decay like $|\bx|^{-1/3}$ \citep[\S 5.4]{Guillod-review2015},
however some compatibility conditions on $\bff$ ensuring the existence
of a solution satisfying \eqref{bound-ubar} with $\bu_{\infty}=\bzero$
are known \citep[\S 3.6]{Guillod-review2015}.

For two-dimensional exterior domains with $\partial\Omega\neq\emptyset$,
we would \emph{a priori} also expect the existence of infinitely many
weak solutions parameterized by some parameter in $\mathbb{R}^{2}$.
However, this question is open and therefore no general weak-strong
uniqueness result comparable to \thmref{uniqueness} is known if $\partial\Omega\neq\emptyset$.
We remark that the method of proof used here for $\Omega=\mathbb{R}^{2}$
does not work if $\partial\Omega\neq\emptyset$, and that it is even
not clear if the mean $\blambda\in\mathbb{R}^{2}$ will furnish a
parametrization in this case.

The asymptotic behavior of the weak solutions in $\Omega=\mathbb{R}^{2}$,
can obviously be determined when our weak-strong theorem is applicable,
but otherwise, we are not able to prove more than the best currently
known results of \citet{Gilbarg.Weinberger-AsymptoticPropertiesof1974,Gilbarg.Weinberger-Asymptoticpropertiesof1978}.
The result of \citet{Amick-Leraysproblemsteady1988} cannot be used
to prove the boundedness of the weak solutions, due to the fact that
the maximum principle used in the proof does not hold on the region
where $\bff$ has support.

For $\Omega=\mathbb{R}^{3}$ and at any fixed force term $\bff$,
we expect the map $\bu_{\infty}\in\mathbb{R}^{3}\mapsto\blambda\in\mathbb{R}^{3}$
to be multivalued since nonuniqueness is expected for large data.
Moreover, it is not clear if this map is surjective. In two dimensions,
we might speculate the existence of a multivalued map $\blambda\in\mathbb{R}^{2}\mapsto\bu_{\infty}\in\mathbb{R}^{2}$
at fixed forcing $\bff$, even if the asymptotic behavior of the weak
solutions is unknown. However, it is not clear if one can find a nontrivial
forcing $\bff$, such that for any $\bu_{\infty}\in\mathbb{R}^{2}$
a weak solution $\tilde{\bu}$ satisfying the hypotheses of \thmref{uniqueness}
can be proven. Therefore, we can not prove that the mapping $\blambda\in\mathbb{R}^{2}\mapsto\bu_{\infty}\in\mathbb{R}^{2}$
is well-defined even for one nontrivial $\bff$ (when $\bff=\bzero$,
the mapping is trivially the identity). Even if this could be proven,
this is not clear if this well-defined map will be injective or surjective.

\section{Function spaces}

We first start with the following standard generalization of the Poincaré
inequality, see for example \citet[Theorems 1.5 \& 1.9]{Necas-DirectMethods2012}:
\begin{lem}
\label{lem:generalized-poincare}Let $\Omega\subset\mathbb{R}^{2}$
be a bounded Lipschitz domain and $\lambda$ a subset of positive
measure of either $\Omega$ or $\partial\Omega$. Then, there exists
$C>0$ depending on $\Omega$ and $\lambda$ such that
\[
\bigl\Vert\bu\bigr\Vert_{L^{2}(\Omega)}\leq C\left(\bigl\Vert\bnabla\bu\bigr\Vert_{L^{2}(\Omega)}+\left|\fint_{\lambda}\bu\right|\right)\,,
\]
for all $\bu\in\dot{H}^{1}(\Omega)$.\end{lem}
\begin{proof}
First we note that if $\bu\in\dot{H}^{1}(\Omega)$, then by the standard
Poincaré inequality $\bu\in H^{1}(\Omega)$, so $\bu\in L^{1}(\lambda)$
and the mean over $\lambda$ is well-defined. We use a proof by contradiction.
If the inequality is false, we can find a sequence $\left(\bu_{n}\right)_{n\in\mathbb{N}}\in H^{1}(\Omega)$
such that $\Vert\bu_{n}\Vert_{L^{2}(\Omega)}=1$ and
\[
\bigl\Vert\bnabla\bu_{n}\bigr\Vert_{L^{2}(\Omega)}+\left|\fint_{\lambda}\bu_{n}\right|<\frac{1}{n}\,.
\]
Since $H^{1}(\Omega)$ is compactly embedded in $L^{2}(\Omega)$,
we can find a subsequence also denoted by $\left(\bu_{n}\right)_{n\in\mathbb{N}}$
and $\bu\in H^{1}(\Omega)$ such that $\bu_{n}\rightharpoonup\bu$
weakly in $H^{1}(\Omega)$ and $\bu_{n}\to\bu$ strongly in $L^{2}(\Omega)$.
Therefore,
\[
\bigl\Vert\bnabla\bu\bigr\Vert_{L^{2}(\Omega)}\leq\liminf_{n\to\infty}\bigl\Vert\bnabla\bu_{n}\bigr\Vert_{L^{2}(\Omega)}=0\,,
\]
so $\bu_{n}\to\bu$ strongly in $H^{1}(\Omega)$ and $\bu$ is a constant.
We can show that
\[
\fint_{\lambda}\bu=\lim_{n\to\infty}\fint_{\lambda}\bu_{n}=\bzero\,,
\]
and since $\lambda$ has positive measure and $\Omega$ is connected,
we obtain $\bu=\bzero$, in contradiction to $\left\Vert \bu\right\Vert _{L^{2}(\Omega)}=1$.
\end{proof}
\pagebreak{}In a second step, we determine a generalized Hardy inequality:
\begin{lem}
\label{lem:generalized-hardy}Let $\Omega\subset\mathbb{R}^{2}$ be
an exterior domain having a compact connected Lipschitz boundary (in
particular $\Omega=\mathbb{R}^{2}$ is allowed), and let $\lambda$
denote a bounded subset of positive measure of either $\Omega$ or
$\partial\Omega$. There exists a constant $C>0$ depending only on
$\Omega$ and $\lambda$ such that
\[
\bigl\Vert\bu\weight\bigr\Vert_{L^{2}(\Omega)}\leq C\left(\bigl\Vert\bnabla\bu\bigr\Vert_{L^{2}(\Omega)}+\fint_{\lambda}\bu\right)\,,
\]
for all $\bu\in\dot{H}^{1}(\Omega)$, where
\begin{align*}
\weight(\bx) & =\frac{1}{\langle\bx\rangle\langle\log\langle\bx\rangle\rangle}\,, & \langle\bx\rangle & =1+|\bx|\,.
\end{align*}
\end{lem}
\begin{proof}
Let $R>0$ be such that $\mathbb{R}^{2}\setminus\Omega\subset B_{R}$
and $\lambda\subset B_{R}$. In this proof $C$ denotes a positive
constant depending only on $\lambda$ and $R$, but which might change
from line to line. Let $\chi$ be a smooth radial cutoff function
such that $\chi(\bx)=1$ if $\bx\in B_{R}$ and $\chi(\bx)=0$ if
$\bx\notin B_{2R}$. We consider the splitting $\bu=\bu_{1}+\bu_{2}$,
where $\bu_{1}=\chi\bu$ and $\bu_{2}=(1-\chi)\bu$. By using the
generalized Poincaré inequality of \lemref{generalized-poincare},
we first remark that
\[
\bigl\Vert\bu\bigr\Vert_{L^{2}(B_{2R})}\leq C\left(\bigl\Vert\bnabla\bu\bigr\Vert_{L^{2}(B_{2R})}+\left|\fint_{\lambda}\bu\right|\right)\,.
\]
For the first part, we have
\[
\bigl\Vert\bu_{1}\weight\bigr\Vert_{L^{2}(\Omega)}=\bigl\Vert\chi\bu\weight\bigr\Vert_{L^{2}(B_{2R})}\leq\bigl\Vert\chi\weight\bigr\Vert_{L^{\infty}(B_{2R})}\bigl\Vert\bu\bigr\Vert_{L^{2}(B_{2R})}\leq C\left(\bigl\Vert\bnabla\bu\bigr\Vert_{L^{2}(B_{2R})}+\left|\fint_{\lambda}\bu\right|\right)\,.
\]
For the second part, we first recall the following standard Hardy
inequality, 
\[
\left\Vert \frac{\bu}{\left|\bx\right|\log(R^{-1}\left|\bx\right|)}\right\Vert _{L^{2}(\Omega\setminus B_{R})}\leq\frac{2}{R}\bigl\Vert\bnabla\bu\bigr\Vert_{L^{2}(\Omega\setminus B_{R})}\,,
\]
valid for all $\bu\in H^{1}(\Omega\setminus B_{R})$ having vanishing
trace of $\partial B_{R}$, see for example \citet[Theorem II.6.1]{Galdi-IntroductiontoMathematical2011}.
Since there exists $C>0$ such that
\[
\weight(\bx)=\frac{1}{\langle\bx\rangle\langle\log\langle\bx\rangle\rangle}\leq\frac{C}{\left|\bx\right|\log(R^{-1}\left|\bx\right|)}\,,
\]
for $\left|\bx\right|>R$, we obtain 
\[
\bigl\Vert\bu_{2}\weight\bigr\Vert_{L^{2}(\Omega)}=\bigl\Vert\bu_{2}\weight\bigr\Vert_{L^{2}(\Omega\setminus B_{R})}\leq C\bigl\Vert\bnabla\bu_{2}\bigr\Vert_{L^{2}(\Omega\setminus B_{R})}\,.
\]
Since $\bnabla\bu_{2}=(1-\chi)\bnabla\bu-\bnabla\chi\otimes\bu$,
we have 
\begin{align*}
\bigl\Vert\bnabla\bu_{2}\bigr\Vert_{L^{2}(\Omega\setminus B_{R})} & \leq\bigl\Vert(1-\chi)\bnabla\bu\bigr\Vert_{L^{2}(\Omega\setminus B_{R})}+\bigl\Vert\bnabla\chi\otimes\bu\bigr\Vert_{L^{2}(B_{2R})}\\
 & \leq\bigl\Vert1-\chi\bigr\Vert_{L^{\infty}(\Omega\setminus B_{R})}\bigl\Vert\bnabla\bu\bigr\Vert_{L^{2}(\Omega\setminus B_{R})}+\bigl\Vert\bnabla\chi\bigr\Vert_{L^{\infty}(B_{2R})}\bigl\Vert\bu\bigr\Vert_{L^{2}(B_{2R})}\\
 & \leq C\bigl\Vert\bnabla\bu\bigr\Vert_{L^{2}(\Omega\setminus B_{R})}+C\left(\bigl\Vert\bnabla\bu\bigr\Vert_{L^{2}(B_{2R})}+\left|\fint_{\lambda}\bu\right|\right)\,.
\end{align*}
Therefore, putting all the bounds together, we have
\[
\bigl\Vert\bu\weight\bigr\Vert_{L^{2}(\Omega)}\leq\bigl\Vert\bu_{1}\weight\bigr\Vert_{L^{2}(\Omega)}+\bigl\Vert\bu_{2}\weight\bigr\Vert_{L^{2}(\Omega)}\leq C\left(\bigl\Vert\bnabla\bu\bigr\Vert_{L^{2}(\Omega)}+\left|\fint_{\lambda}\bu\right|\right)\,,
\]
and the lemma is proven.
\end{proof}
In view of the result of \lemref{generalized-poincare,generalized-hardy}
with $\lambda=\partial\Omega$, we see that the semi-norm of $\dot{H}^{1}(\Omega)$
defines a norm on $C_{0}^{\infty}(\Omega)$ if $\partial\Omega\neq\emptyset$.
Therefore, we have the following standard result, see for example
\citet{Galdi-IntroductiontoMathematical2011} or \citet{Sohr-Navier-Stokesequations.elementary2001}:
\begin{prop}
\label{prop:completion-gamma}Let $\Omega\subset\mathbb{R}^{2}$ be
an exterior domain having a compact connected Lipschitz boundary $\partial\Omega\neq\emptyset$.
Then the completion of $C_{0,\sigma}^{\infty}(\Omega)$ in the norm
of $\dot{H}^{1}(\Omega)$ is the Hilbert space
\[
\dot{H}_{0,\sigma}^{1}(\Omega)=\left\{ \bu\in\dot{H}_{\sigma}^{1}(\Omega):\,\Gamma_{\partial\Omega}\bu=\bzero\right\} \,,
\]
with the inner product 
\[
\bigl\langle\bu,\bv\bigr\rangle_{\dot{H}_{0,\sigma}^{1}(\Omega)}=\bigl\langle\bnabla\bu,\bnabla\bv\bigr\rangle_{L^{2}(\Omega)}\,.
\]
Moreover, $\dot{H}_{0,\sigma}^{1}(\Omega)$ has the following equivalent
norms,
\[
\bigl\Vert\bu\bigr\Vert_{L^{2}(\Omega\cap B_{R})}+\bigl\Vert\bnabla\bu\bigr\Vert_{L^{2}(\Omega)}\,,
\]
for any $R>0$ such $\partial\Omega\cap B_{R}\neq\emptyset$, and
\[
\bigl\Vert\bu\weight\bigr\Vert_{L^{2}(\Omega)}+\bigl\Vert\bnabla\bu\bigr\Vert_{L^{2}(\Omega)}\,.
\]
\end{prop}
\begin{proof}
The proof that the completion of $C_{0,\sigma}^{\infty}(\Omega)$
in the norm of $\dot{H}^{1}(\Omega)$ is equal to $\dot{H}_{0,\sigma}^{1}(\Omega)$
is given in \citet[Theorems II.7.3 \& III.5.1]{Galdi-IntroductiontoMathematical2011}
or in \citet[Lemma III.1.2.1]{Sohr-Navier-Stokesequations.elementary2001}.
The equivalence of the norms follows from the generalized Poincaré
inequality of \lemref{generalized-poincare} with $\lambda=\partial\Omega\cap B_{R}$
and from \lemref{generalized-hardy}.
\end{proof}
When the boundary is trivial, \emph{i.e.} $\Omega=\mathbb{R}^{2}$,
the boundary cannot be used as an anchor point for the Poincaré inequality
and in particular the semi-norm of $\dot{H}^{1}(\Omega)$ does not
define a norm on $C_{0}^{\infty}(\Omega)$. The idea is to fix some
bounded subset $\omega\subset\Omega$ of positive measure so that
$\dot{H}^{1}(\Omega)$ is an Hilbert space with the inner product
\[
\bigl\langle\bnabla\bu,\bnabla\bv\bigr\rangle_{L^{2}(\Omega)}+\fint_{\omega}\bu\bcdot\fint_{\omega}\bv\,.
\]
Therefore, the following result stays also valid for $\Omega=\mathbb{R}^{2}$
and will play a crucial role in the construction of weak solutions
in $\Omega=\mathbb{R}^{2}$:
\begin{prop}
\label{prop:completion-omega}Let $\Omega\subset\mathbb{R}^{2}$ be
an exterior domain having a compact connected Lipschitz boundary (in
particular $\Omega=\mathbb{R}^{2}$ is allowed). Given a bounded subset
$\omega\subset\Omega$ of positive measure, the completion of
\[
C_{0,\sigma}^{\infty}(\Omega,\omega)=\left\{ \bphi\in C_{0,\sigma}^{\infty}(\Omega):\,\fint_{\omega}\bphi=\bzero\right\} ,
\]
in the norm of $\dot{H}^{1}(\Omega)$ is the Hilbert space
\[
\dot{H}_{0,\sigma}^{1}(\Omega,\omega)=\left\{ \bu\in\dot{H}_{\sigma}^{1}(\Omega):\,\Gamma_{\partial\Omega}\bu=\bzero\quad\text{and}\quad\fint_{\omega}\bu=\bzero\right\} \,,
\]
with the inner product 
\[
\bigl\langle\bu,\bv\bigr\rangle_{\dot{H}_{0,\sigma}^{1}(\Omega,\omega)}=\bigl\langle\bnabla\bu,\bnabla\bv\bigr\rangle_{L^{2}(\Omega)}\,.
\]
Moreover, $\dot{H}_{0,\sigma}^{1}(\Omega,\omega)$ has the following
equivalent norms,
\[
\bigl\Vert\bu\bigr\Vert_{L^{2}(\Omega\cap B_{R})}+\bigl\Vert\bnabla\bu\bigr\Vert_{L^{2}(\Omega)}\,,
\]
for any $R>0$ such that $\omega\subset B_{R}$, and
\[
\bigl\Vert\bu\weight\bigr\Vert_{L^{2}(\Omega)}+\bigl\Vert\bnabla\bu\bigr\Vert_{L^{2}(\Omega)}\,.
\]
\end{prop}
\begin{proof}
Let $\dot{H}_{0,\sigma}^{1}(\Omega,\omega)$ denote the completion
of $C_{0,\sigma}^{\infty}(\Omega,\omega)$ in the norm of $\dot{H}^{1}(\Omega)$.
First of all we remark that $\dot{H}_{0,\sigma}^{1}(\Omega,\omega)\subset\left\{ \bu\in\dot{H}_{0,\sigma}^{1}(\Omega):\,\fint_{\omega}\bu=\bzero\right\} $.
Using the generalized Poincaré and Hardy inequalities (\lemref{generalized-poincare,generalized-hardy}),
we have
\[
\bigl\Vert\bu\bigr\Vert_{L^{2}(\Omega\cap B_{R})}^{2}\leq C\left(\bigl\Vert\bnabla\bu\bigr\Vert_{L^{2}(\Omega\cap B_{R})}^{2}+\left|\fint_{\omega}\bu\right|^{2}\right)\,,
\]
and
\[
\bigl\Vert\bu\weight\bigr\Vert_{L^{2}(\Omega)}\leq C\left(\bigl\Vert\bnabla\bu\bigr\Vert_{L^{2}(\Omega)}+\fint_{\omega}\bu\right)\,,
\]
for any $\bu\in\dot{H}^{1}(\Omega)$, which show the claimed equivalence
of the norms. Therefore, it only remains to prove that any $\bu\in\dot{H}_{0,\sigma}^{1}(\Omega,\omega)$
can be approximated by functions in $C_{0,\sigma}^{\infty}(\Omega,\omega)$.
The proof of this fact follows almost directly by using the proofs
presented in Chapters II \& III of \citet{Galdi-IntroductiontoMathematical2011},
so we only sketch the main steps.

Let $\psi:\mathbb{R}^{+}\to[0,1]$ be a smooth cutoff function such
that $\psi(r)=1$ if $r\leq1/2$ and $\psi(r)=0$ if $r\geq1$. For
$n>0$ large enough, then
\[
\psi_{n}(\bx)=\psi\biggl(\frac{\log\langle\log\langle\bx\rangle\rangle}{\log\langle\log\langle n\rangle\rangle}\biggr)\,,
\]
is a cutoff function such that $\psi_{n}(\bx)=0$ if $\left|\bx\right|\geq n$
and $\psi_{n}(\bx)=1$ if $\left|\bx\right|\leq\gamma_{n}$ where
\[
\gamma_{n}=\exp\left(\sqrt{\langle\log\langle n\rangle\rangle}-1\right)-1\,.
\]
Explicitly, we have
\begin{equation}
\left|\bnabla\psi_{n}(\bx)\right|\leq\frac{\left\Vert \psi^{\prime}\right\Vert _{\infty}}{\log\langle\log\langle n\rangle\rangle}\weight(\bx)\,,\label{eq:psi-grad}
\end{equation}

Therefore $\psi_{n}\bu$ has compact support, vanishing mean on $\omega$,
belongs to $H^{1}(\Omega)$ and converges to $\bu$ in $\dot{H}^{1}(\Omega)$
as $n\to\infty$ by using \eqref{psi-grad} and applying \lemref{generalized-hardy}
(see \citealp{Galdi-IntroductiontoMathematical2011}, Theorems II.7.1
\& II.7.2). Moreover, $\psi_{n}\bu$ is divergence-free except on
the annulus $\gamma_{n}\leq\left|\bx\right|\leq n$. There exists
a corrector $\bw_{n}\in\dot{H}^{1}(\Omega)$ having support in the
annulus $\gamma_{n}\leq\left|\bx\right|\leq n$ such that $\psi_{n}\bu+\bw_{n}$
is divergence-free and $\left\Vert \bw_{n}\right\Vert _{\dot{H}^{1}(\Omega)}\leq C\left\Vert \bu\bcdot\bnabla\psi_{n}\right\Vert _{L^{2}(\Omega)}$
with $C>0$ independent of $n$ (see \citealp{Galdi-IntroductiontoMathematical2011},
Theorem III.3.1). Therefore, $\psi_{n}\bu+\bw_{n}$ has support in
$B_{n}$, zero mean on $\omega$, vanishing trace on $\partial\Omega$,
belongs to $\dot{H}_{\sigma}^{1}(\Omega)$ and converges to $\bu$
in $\dot{H}_{\sigma}^{1}(\Omega)$ by \eqref{psi-grad} and \lemref{generalized-hardy}.
Now for any $n>0$, there exists a smoothing $\bu_{n}\in C_{0,\sigma}^{\infty}(\Omega)$
of $\psi_{n}\bu+\bw_{n}$ such that
\[
\bigl\Vert\psi_{n}\bu+\bw_{n}-\bu_{n}\bigr\Vert_{\dot{H}^{1}(\Omega)}+\bigl\Vert\psi_{n}\bu+\bw_{n}-\bu_{n}\bigr\Vert_{L^{2}(\Omega\cap B_{n})}\leq\frac{1}{n}\,.
\]
(see \citealp{Galdi-IntroductiontoMathematical2011}, Theorems III.4.1
\& III.4.2). Hence we have
\[
\left|\fint_{\omega}\bu_{n}\right|=\left|\fint_{\omega}(\bu_{n}-\psi_{n}\bu)\right|\leq\fint_{\omega}\left|\bu_{n}-\psi_{n}\bu\right|\leq\left|\omega\right|^{-1/2}\left\Vert \psi_{n}\bu-\bu_{n}\right\Vert _{L^{2}(\omega)}\leq\frac{1}{\left|\omega\right|^{1/2}n}\,.
\]
Finally, it is not hard to find two explicit functions $\bv_{i}\in C_{0,\sigma}^{\infty}(\Omega)$
such that $\fint_{\omega}\bv_{i}=\be_{i}$ for $i=1,2$. Therefore
$\bu_{n}+\left(\be_{1}\otimes\bv_{1}+\be_{2}\otimes\bv_{2}\right)\bcdot\fint_{\omega}\bu_{n}\in C_{0,\sigma}^{\infty}(\Omega,\omega)$
converges to $\bu$ in $\dot{H}_{\sigma}^{1}(\Omega,\omega)$ as $n\to\infty$.
\end{proof}
Finally, we discuss conditions on which $\bff$ can be represented
as $\bff=\bnabla\bcdot\bbF$ with $\bbF\in L^{2}(\Omega)$ and in
particular we prove the claims made in \remref{exterior-L2-compact,R2-L2-compact}.
\begin{lem}
\label{lem:exterior-representation}Let $\Omega\subset\mathbb{R}^{2}$
be an exterior domain having a compact connected Lipschitz boundary
$\partial\Omega\neq\emptyset$. Let $\bff\in L_{\loc}^{1}(\Omega)$.
If the linear form $\bphi\mapsto\bigl\langle\bff,\bphi\bigr\rangle_{L^{2}(\Omega)}$
is continuous on $\dot{H}_{0,\sigma}^{1}(\Omega)$, then there exists
$\bF\in L^{2}(\Omega)$ such that $\bff=\bnabla\bcdot\bF$ in the
following sense:
\[
\bigl\langle\bff,\bphi\bigr\rangle_{L^{2}(\Omega)}=-\bigl\langle\bbF,\bnabla\bphi\bigr\rangle_{L^{2}(\Omega)}\,,
\]
for all $\bphi\in C_{0,\sigma}^{\infty}(\Omega)$. In particular this
holds when $\bff/\weight\in L^{2}(\Omega)$.\end{lem}
\begin{proof}
By using Riesz representation theorem, there exists $\bu\in\dot{H}_{0,\sigma}^{1}(\Omega)$
such that
\[
\bigl\langle\bnabla\bu,\bnabla\bphi\bigr\rangle_{L^{2}(\Omega)}=\bigl\langle\bff,\bphi\bigr\rangle_{L^{2}(\Omega)}\,,
\]
for all $\bphi\in\dot{H}_{0}^{1}(\Omega)$ and we can take $\bbF=\bnabla\bu$.
If $\bff/\weight\in L^{2}(\Omega)$, then by \lemref{generalized-hardy}
with $\lambda=\partial\Omega$, we have
\[
\Bigl|\bigl\langle\bff,\bphi\bigr\rangle_{L^{2}(\Omega)}\Bigr|\leq\bigl\Vert\bff/\weight\bigr\Vert_{L^{2}(\Omega)}\bigl\Vert\bphi\weight\bigr\Vert_{L^{2}(\Omega)}\leq C\bigl\Vert\bnabla\bphi\bigr\Vert_{L^{2}(\Omega)}\,,
\]
so the linear form is continuous on $\dot{H}_{0}^{1}(\Omega)$.\end{proof}
\begin{lem}
\label{lem:R2-representation}Let $\Omega\subset\mathbb{R}^{2}$ be
an exterior domain having a compact connected Lipschitz boundary (in
particular $\Omega=\mathbb{R}^{2}$ is allowed). Let $\bff\in L^{1}(\Omega)$.
If the linear form $\bphi\mapsto\bigl\langle\bff,\bphi\bigr\rangle_{L^{2}(\Omega)}$
is continuous on $\dot{H}_{0,\sigma}^{1}(\Omega,\omega)$ and $\int_{\Omega}\bff=\bzero$,
then there exists $\bF\in L^{2}(\Omega)$ such that $\bff=\bnabla\bcdot\bF$
in the following sense:
\[
\bigl\langle\bff,\bphi\bigr\rangle_{L^{2}(\Omega)}=-\bigl\langle\bbF,\bnabla\bphi\bigr\rangle_{L^{2}(\Omega)}\,,
\]
for all $\bphi\in C_{0,\sigma}^{\infty}(\Omega)$. In particular this
holds when $\bff/\weight\in L^{2}(\Omega)$ and $\int_{\Omega}\bff=\bzero$.\end{lem}
\begin{proof}
By using Riesz representation theorem, there exists $\bu\in\dot{H}_{0,\sigma}^{1}(\Omega,\omega)$
such that
\[
\bigl\langle\bnabla\bu,\bnabla\bpsi\bigr\rangle_{L^{2}(\Omega)}=\bigl\langle\bff,\bpsi\bigr\rangle_{L^{2}(\Omega)}\,,
\]
for all $\bpsi\in\dot{H}_{0}^{1}(\Omega,\omega)$. For any $\bphi\in C_{0,\sigma}^{\infty}(\Omega)$,
let $\bpsi=\bphi-\bar{\bphi}\in\dot{H}_{0}^{1}(\Omega,\omega)$ and
therefore
\[
\bigl\langle\bnabla\bu,\bnabla\bphi\bigr\rangle_{L^{2}(\Omega)}=\bigl\langle\bff,\bpsi\bigr\rangle_{L^{2}(\Omega)}=\bigl\langle\bff,\bphi\bigr\rangle_{L^{2}(\Omega)}
\]
because $\int_{\Omega}\bff=\bzero$. If in addition $\bff/\weight\in L^{2}(\Omega)$,
then by \lemref{generalized-hardy} with $\lambda=\omega$, we have
\[
\Bigl|\bigl\langle\bff,\bpsi\bigr\rangle_{L^{2}(\Omega)}\Bigr|\leq\bigl\Vert\bff/\weight\bigr\Vert_{L^{2}(\Omega)}\bigl\Vert\bpsi\weight\bigr\Vert_{L^{2}(\Omega)}\leq C\bigl\Vert\bnabla\bpsi\bigr\Vert_{L^{2}(\Omega)}\,,
\]
for any $\bpsi\in\dot{H}_{0}^{1}(\Omega,\omega)$.\end{proof}
\begin{rem}
The hypothesis $\int_{\Omega}\bff=\bzero$ is needed only for $\Omega=\mathbb{R}^{2}$
and not if $\partial\Omega\neq\emptyset$. This fact is linked to
the Stokes paradox, since the existence proof given below works equally
well for the Stokes equation. For $\Omega=\mathbb{R}^{2}$, it is
well known that the Stokes equations have a solution in $\dot{H}_{\sigma}^{1}(\Omega)$
if and only if $\int_{\Omega}\bff=\bzero$. Otherwise, the solutions
of the Stokes equations in $\Omega=\mathbb{R}^{2}$ grow like $\log\left|\bx\right|$
at infinity, hence the Stokes equations have no solutions in $\dot{H}_{\sigma}^{1}(\Omega)$.
If $\Omega\neq\mathbb{R}^{2}$, the Stokes equations always admit
a solution in $\dot{H}_{\sigma}^{1}(\Omega)$ regardless of the mean
of $\bff$.
\end{rem}

\section{Proof of existence}

The main idea to construct weak solutions in $\Omega=\mathbb{R}^{2}$
is to construct for each $n\in\mathbb{N}$ large enough a particular
weak solution in the ball $B_{n}$ having a prescribed mean on a bounded
subset of positive measure $\omega\subset\Omega$. This can be done
be choosing a suitable constant $\bc_{n}$ on the artificial boundary
$\partial B_{n}$.
\begin{prop}
\label{prop:approximate-weak-solution}Assume that the hypotheses
of \thmref{weak-solution} hold. For any $\blambda\in\mathbb{R}^{2}$
and $n\in\mathbb{N}$ large enough such that $\omega\subset B_{n}$,
there exists $\bc_{n}\in\mathbb{R}^{2}$ and a weak solution $\bu_{n}\in\dot{H}_{\sigma}^{1}(B_{n})$
of the Navier\textendash Stokes equations \eqref{intro-ns-eq} in
$B_{n}$ such that:
\begin{enumerate}
\item $\left.\bu_{n}\right|_{\partial B_{n}}=\blambda+\bc_{n}$ in the trace
sense\,;
\item $\bigl\Vert\bnabla\bu_{n}\bigr\Vert_{L^{2}(B_{n})}=\bigl\langle\bbF,\bnabla\bu_{n}\bigr\rangle_{L^{2}(B_{n})}$\,;
\item $\fint_{\omega}\bu_{n}=\blambda$\,.
\end{enumerate}
\end{prop}
\begin{proof}
For any vector field $\bv\in L_{\loc}^{1}(\omega)$, we denote by
$\bar{\bv}$ the mean of $\bv$ on $\omega$, $\bar{\bv}=\fint_{\omega}\bv=\frac{1}{\left|\omega\right|}\int_{\omega}\bv$.
We look for a solution of the form $\bu_{n}=\blambda+\bv_{n}-\bar{\bv}_{n}$
with $\bv_{n}\in\dot{H}_{0,\sigma}^{1}(B_{n})$ so that the third
condition of the proposition automatically holds. He have $\left.\bu_{n}\right|_{\partial B_{n}}=\blambda-\bar{\bv}_{n}$,
so the first condition is satisfied by choosing $\bc_{n}=-\bar{\bv}_{n}$.
Therefore, it remains to prove the existence of $\bv_{n}\in\dot{H}_{0,\sigma}^{1}(B_{n})$
such that
\begin{equation}
\bigl\langle\bnabla\bv_{n},\bnabla\bphi\bigr\rangle_{L^{2}(B_{n})}+\bigl\langle\bigl(\blambda+\bv_{n}-\bar{\bv}_{n}\bigr)\bcdot\bnabla\bv_{n},\bphi\bigr\rangle_{L^{2}(B_{n})}=\bigl\langle\bbF,\bnabla\bphi\bigr\rangle_{L^{2}(B_{n})}\,,\label{eq:weak-solution-approximate}
\end{equation}
for all $\bphi\in C_{0,\sigma}^{\infty}(B_{n})$.

Since
\[
\Bigl|\bigl\langle\bbF,\bnabla\bphi\bigr\rangle_{L^{2}(B_{n})}\Bigr|\leq\bigl\Vert\bbF\bigr\Vert_{L^{2}(B_{n})}\bigl\Vert\bnabla\bphi\bigr\Vert_{L^{2}(B_{n})}\leq\bigl\Vert\bbF\bigr\Vert_{L^{2}(\Omega)}\bigl\Vert\bphi\bigr\Vert_{\dot{H}_{0,\sigma}^{1}(B_{n})}\,,
\]
for all $\bphi\in\dot{H}_{0,\sigma}^{1}(B_{n})$, by using Riesz representation
theorem, there exists $\bFn\in\dot{H}_{0,\sigma}^{1}(B_{n})$, such
that
\[
\bigl\langle\bFn,\bphi\bigr\rangle_{\dot{H}_{0,\sigma}^{1}(B_{n})}=\bigl\langle\bbF,\bnabla\bphi\bigr\rangle_{L^{2}(B_{n})}\,,
\]
for all $\bphi\in C_{0,\sigma}^{\infty}(\Omega)$.

The bilinear map $\cBn$ defined by
\[
\bigl\langle\cBn(\bv,\bw),\bphi\bigr\rangle_{\dot{H}_{0,\sigma}^{1}(B_{n})}=\bigl\langle(\bv-\bar{\bv})\bcdot\bnabla\bw,\bphi\bigr\rangle_{L^{2}(B_{n})}\,,
\]
is continuous on $L^{4}(B_{n})$, 
\begin{align*}
\Bigl|\bigl\langle\cBn(\bv,\bw),\bphi\bigr\rangle_{\dot{H}_{0,\sigma}^{1}(B_{n})}\Bigr| & \leq\Bigl|\bigl\langle(\bv-\bar{\bv})\bcdot\bnabla\bphi,\bw\bigr\rangle_{L^{2}(B_{n})}\Bigr|\\
 & \leq\left(\bigl\Vert\bv\bigr\Vert_{L^{4}(B_{n})}+\bigl\Vert\bar{\bv}\bigr\Vert_{L^{4}(B_{n})}\right)\bigl\Vert\bw\bigr\Vert_{L^{4}(B_{n})}\bigl\Vert\bphi\bigr\Vert_{\dot{H}_{0,\sigma}^{1}(B_{n})}\\
 & \leq\left(1+\frac{\pi n^{2}}{\left|\omega\right|}\right)\bigl\Vert\bv\bigr\Vert_{L^{4}(B_{n})}\bigl\Vert\bw\bigr\Vert_{L^{4}(B_{n})}\bigl\Vert\bphi\bigr\Vert_{\dot{H}_{0,\sigma}^{1}(B_{n})}\,,
\end{align*}
because
\[
\bigl\Vert\bar{\bv}\bigr\Vert_{L^{4}(B_{n})}\leq\pi^{1/4}n^{1/2}\left|\bar{\bv}\right|\leq\frac{\pi^{1/4}n^{1/2}}{\left|\omega\right|}\int_{B_{n}}|\bv|\leq\frac{\pi n^{2}}{\left|\omega\right|}\bigl\Vert\bv\bigr\Vert_{L^{4}(B_{n})}\,.
\]
The linear map $\cLn$ defined by
\[
\bigl\langle\cLn(\bv),\bphi\bigr\rangle_{\dot{H}_{0,\sigma}^{1}(B_{n})}=\bigl\langle\blambda\bcdot\bnabla\bv,\bphi\bigr\rangle_{L^{2}(B_{n})}\,,
\]
is also continuous on $L^{4}(B_{n})$,
\[
\Bigl|\bigl\langle\cLn(\bv),\bphi\bigr\rangle_{\dot{H}_{0,\sigma}^{1}(B_{n})}\Bigr|\leq\Bigl|\bigl\langle\blambda\bcdot\bnabla\bphi,\bv\bigr\rangle_{L^{2}(B_{n})}\Bigr|\leq\bigl\Vert\blambda\bigr\Vert_{L^{4}(B_{n})}\bigl\Vert\bw\bigr\Vert_{L^{4}(B_{n})}\bigl\Vert\bphi\bigr\Vert_{\dot{H}_{0,\sigma}^{1}(B_{n})}\,.
\]
Therefore, the map $\cAn:\dot{H}_{0,\sigma}^{1}(B_{n})\to\dot{H}_{0,\sigma}^{1}(B_{n})$
defined by $\cAn(\bv)=\cBn(\bv,\bv)+\cLn(\bv)$ is continuous on $\dot{H}_{0,\sigma}^{1}(B_{n})$
when equipped with the $L^{4}$-norm, hence completely continuous
on $\dot{H}_{0,\sigma}^{1}(B_{n})$, since $\dot{H}_{0,\sigma}^{1}(B_{n})$
is compactly embedded in $L^{4}(B_{n})$.

We have
\[
\bigl\langle\bv_{n}+\cAn(\bv_{n})-\bFn,\bphi\bigr\rangle_{\dot{H}_{0,\sigma}^{1}(B_{n})}=\bigl\langle\bnabla\bv_{n},\bnabla\bphi\bigr\rangle_{L^{2}(B_{n})}+\bigl\langle\bigl(\blambda+\bv_{n}-\bar{\bv}_{n}\bigr)\bcdot\bnabla\bv_{n},\bphi\bigr\rangle_{L^{2}(B_{n})}+\bigl\langle\bff,\bphi\bigr\rangle_{L^{2}(B_{n})}\,,
\]
so the weak formulation \eqref{weak-solution-approximate} is equivalent
to the functional equation
\begin{equation}
\bv_{n}+\cAn(\bv_{n})-\bFn=\bzero\label{eq:ns-functional}
\end{equation}
in $\dot{H}_{0,\sigma}^{1}(B_{n})$. From the Leray\textendash Schauder
fixed point theorem \citep[see for example][Theorem 11.6]{Gilbarg.Trudinger-EllipticPartialDifferential1988}
to prove the existence of a solution to \eqref{ns-functional} it
is sufficient to prove that the set of solutions $\bv$ of the equation
\begin{equation}
\bv_{n}+\lambda\left(\cAn(\bv_{n})-\bFn\right)=\bzero\label{eq:ns-functional-lambda}
\end{equation}
is uniformly bounded in $\lambda\in\left[0,1\right]$. To this end,
we take the scalar product of \eqref{ns-functional-lambda} with $\bv_{n}$,
\[
\bigl\langle\bnabla\bv_{n},\bnabla\bv_{n}\bigr\rangle_{L^{2}(B_{n})}+\lambda\bigl\langle\bigl(\blambda+\bv_{n}-\bar{\bv}_{n}\bigr)\bcdot\bnabla\bv_{n},\bv_{n}\bigr\rangle_{L^{2}(B_{n})}=\lambda\bigl\langle\bbF,\bnabla\bv_{n}\bigr\rangle_{L^{2}(B_{n})}\,.
\]
By integrating by parts, we obtain
\[
\bigl\langle\bnabla\bv_{n},\bnabla\bv_{n}\bigr\rangle_{L^{2}(B_{n})}=\lambda\bigl\langle\bbF,\bnabla\bv_{n}\bigr\rangle_{L^{2}(B_{n})}\,,
\]
so 
\[
\bigl\Vert\bnabla\bv_{n}\bigr\Vert_{L^{2}(B_{n})}\leq\bigl\Vert\bF\bigr\Vert_{L^{2}(B_{n})}\leq\bigl\Vert\bF\bigr\Vert_{L^{2}(\Omega)}\,.
\]

\end{proof}
Now we can prove the existence of weak solutions in $\Omega=\mathbb{R}^{2}$
by using the method of invading domains:
\begin{proof}[Proof of \thmref{weak-solution}]
By \propref{approximate-weak-solution}, for any $n\in\mathbb{N}$,
there exists $\bc_{n}\in\mathbb{R}^{2}$ and a weak solution $\bu_{n}\in\dot{H}_{\sigma}^{1}(B_{n})$
satisfying the three conditions of this proposition. We write $\bu_{n}=\blambda+\bv_{n}$,
so extending $\bv_{n}$ to $\Omega$ by setting $\bv_{n}=\bc_{n}$
on $\Omega\setminus B_{n}$, we have
\begin{align*}
\bigl\Vert\bnabla\bv_{n}\bigr\Vert_{L^{2}(\Omega)} & =\bigl\langle\bbF,\bnabla\bv_{n}\bigr\rangle_{L^{2}(\Omega)}\,, & \fint_{\omega}\bv_{n} & =\bzero\,,
\end{align*}
and $\left(\bv_{n}\right)_{n\in\mathbb{N}}$ is bounded by $\left\Vert \bbF\right\Vert _{L^{2}(\Omega)}$
in the function space $\dot{H}_{0,\sigma}^{1}(\Omega,\omega)$ defined
by \propref{completion-omega}. Therefore, there exists a subsequence
also denoted by $\left(\bv_{n}\right)_{n\in\mathbb{N}}$ which converges
weakly to $\bv\in\dot{H}_{0,\sigma}^{1}(\Omega,\omega)$. Let $\bu=\blambda+\bv$.
We directly obtain that
\[
\bigl\Vert\bnabla\bu\bigr\Vert_{L^{2}(\Omega)}^{2}=\bigl\Vert\bnabla\bv\bigr\Vert_{L^{2}(\Omega)}^{2}\leq\liminf_{n\to\infty}\bigl\Vert\bnabla\bv_{n}\bigr\Vert_{L^{2}(\Omega)}^{2}\,,
\]
and
\[
\lim_{n\to\infty}\bigl\langle\bbF,\bnabla\bv_{n}\bigr\rangle_{L^{2}(\Omega)}=\bigl\langle\bbF,\bnabla\bv\bigr\rangle_{L^{2}(\Omega)}=\bigl\langle\bbF,\bnabla\bu\bigr\rangle_{L^{2}(\Omega)}\,,
\]
so the energy inequality \eqref{energy-inequality} is proven.

We now prove that the limit $\bu$ is a weak solution to the Navier\textendash Stokes
equations in $\Omega$. Let $\bphi\in C_{0,\sigma}^{\infty}(\Omega)$.
There exists $m\in\mathbb{N}$ such that the support of $\bphi$ is
contained in $B_{m}$. In view of \propref{completion-omega}, $\left(\bv_{n}\right)_{n\in\mathbb{N}}$
is bounded in $H^{1}(B_{m})$, so there exists a subsequence also
denoted by $\left(\bv_{n}\right)_{n\in\mathbb{N}}$ which converging
strongly to $\bv$ in $L^{4}(B_{m})$, since $H^{1}(B_{m})$ is compactly
embedded in $L^{4}(B_{m})$. Since $\bu_{n}=\blambda+\bv_{n}$ is
a weak solution in $B_{n}$, we have
\[
\bigl\langle\bnabla\bu_{n},\bnabla\bphi\bigr\rangle_{L^{2}(B_{m})}+\bigl\langle\bu_{n}\bcdot\bnabla\bu_{n},\bphi\bigr\rangle_{L^{2}(B_{m})}=\bigl\langle\bbF,\bnabla\bphi\bigr\rangle_{L^{2}(B_{m})}\,,
\]
for any $n\geq m$ and it only remains to show that this equation
remains valid in the limit $n\to\infty$. Let $\bpsi=\bphi-\fint_{\omega}\bphi$,
where by \propref{completion-omega}, $\bpsi\in\dot{H}_{0,\sigma}^{1}(\Omega,\omega)$.
By definition of the weak convergence, 
\[
\lim_{n\to\infty}\bigl\langle\bnabla\bu_{n},\bnabla\bphi\bigr\rangle_{L^{2}(B_{m})}=\lim_{n\to\infty}\bigl\langle\bv_{n},\bpsi\bigr\rangle_{\dot{H}_{0,\sigma}^{1}(\Omega,\omega)}=\bigl\langle\bv,\bpsi\bigr\rangle_{\dot{H}_{0,\sigma}^{1}(\Omega,\omega)}=\bigl\langle\bnabla\bu,\bnabla\bphi\bigr\rangle_{L^{2}(B_{m})}\,.
\]
Since $\bphi$ has compact support in $B_{m}$, we have 
\begin{align*}
\Bigl|\bigl\langle\bu_{n}\bcdot\bnabla\bu_{n}-\bu\bcdot\bnabla\bu,\bphi\bigr\rangle_{L^{2}(B_{m})}\Bigr| & \leq\Bigl|\bigl\langle\left(\bu_{n}-\bu\right)\bcdot\bnabla\bu_{n},\bphi\bigr\rangle_{L^{2}(B_{m})}\Bigr|+\Bigl|\bigl\langle\bu\bcdot\left(\bnabla\bu_{n}-\bnabla\bu\right),\bphi\bigr\rangle_{L^{2}(B_{m})}\Bigr|\\
 & \leq\Bigl|\bigl\langle\left(\bv_{n}-\bv\right)\bcdot\bnabla\bv_{n},\bphi\bigr\rangle_{L^{2}(B_{m})}\Bigr|+\Bigl|\bigl\langle\bu\bcdot\bnabla\bphi,\bv_{n}-\bv\bigr\rangle_{L^{2}(B_{m})}\Bigr|\\
 & \leq\left(\bigl\Vert\bnabla\bv_{n}\bigr\Vert_{L^{2}(B_{m})}\bigl\Vert\bphi\bigr\Vert_{L^{4}(B_{m})}+\bigl\Vert\bu\bigr\Vert{}_{L^{4}(B_{m})}\bigl\Vert\bnabla\bphi\bigr\Vert_{L^{4}(B_{m})}\right)\bigl\Vert\bv_{n}-\bv\bigr\Vert_{L^{4}(B_{m})}\,,
\end{align*}
so 
\[
\lim_{n\to\infty}\bigl(\bu_{n}\bcdot\bnabla\bu_{n},\bphi\bigr)=\bigl(\bu\bcdot\bnabla\bu,\bphi\bigr)\,,
\]
and $\bu$ satisfies \eqref{weak-solution}.
\end{proof}

\section{Proof of uniqueness}

We first start with the following approximation lemma:
\begin{lem}
\label{lem:ubar-approx}For $\Omega=\mathbb{R}^{2}$, if $\tilde{\bv}\in\dot{H}_{\sigma}^{1}(\Omega)$
satisfies $\tilde{\bv}/\weight\in L^{\infty}(\Omega)$, then there
exists a sequence $\left(\tilde{\bv}_{n}\right)_{n\in\mathbb{N}}\subset C_{0,\sigma}^{\infty}(\Omega)$
such that $\tilde{\bv}_{n}\to\tilde{\bv}$ strongly in $\dot{H}_{\sigma}^{1}(\Omega)$
and $\bu\otimes\tilde{\bv}_{n}\to\bu\otimes\tilde{\bv}$ strongly
in $L^{2}(\Omega)$ for any $\bu\in\dot{H}_{\sigma}^{1}(\Omega)$.\end{lem}
\begin{proof}
First of all we need a better Sobolev cut-off than the one used in
the proof of \propref{completion-omega}. Let $\eta:\mathbb{R}^{+}\to[0,1]$
be a smooth cutoff function such that $\eta(r)=1$ if $r\leq1/2$
and $\eta(r)=0$ if $r\geq1$. For $n>0$ large enough, then
\[
\eta_{n}(\bx)=\eta\biggl(\frac{\log\langle\log\langle\log\langle\bx\rangle\rangle\rangle}{\log\langle\log\langle\log\langle n\rangle\rangle\rangle}\biggr)\,,
\]
is a cutoff function such that $\eta_{n}(\bx)=0$ if $\left|\bx\right|\geq n$
and $\eta_{n}(\bx)=1$ if $\left|\bx\right|\leq\gamma_{n}$ where
\[
\gamma_{n}=\exp\left(\exp\left(\sqrt{\langle\log\langle\log\langle n\rangle\rangle\rangle}-1\right)-1\right)-1\,.
\]
Explicitly, we have
\begin{equation}
\bigl|\bnabla\eta_{n}(\bx)\bigr|\leq\frac{\left\Vert \eta^{\prime}\right\Vert _{\infty}}{\log\langle\log\langle\log\langle n\rangle\rangle\rangle}\frac{1}{\langle\bx\rangle\langle\log\langle\bx\rangle\rangle\langle\log\langle\log\langle\bx\rangle\rangle\rangle}\,,\label{eq:eta-grad}
\end{equation}
and
\begin{equation}
\bigl|\bnabla^{2}\eta_{n}(\bx)\bigr|\leq\frac{4\left\Vert \eta^{\prime}\right\Vert _{\infty}+2\left\Vert \eta^{\prime\prime}\right\Vert _{\infty}}{\log\langle\log\langle\log\langle n\rangle\rangle\rangle}\frac{1}{\langle\bx\rangle^{2}\langle\log\langle\bx\rangle\rangle\langle\log\langle\log\langle\bx\rangle\rangle\rangle}\,.\label{eq:eta-grad2}
\end{equation}

We define the stream function associated to $\tilde{\bv}$ by the
following curvilinear integral,
\[
\tilde{\psi}(\bx)=\int_{\bzero}^{\bx}\tilde{\bv}^{\perp}\bcdot\rd\bx\,,
\]
so since $\tilde{\bv}/\weight\in L^{\infty}(\Omega)$, we have
\begin{equation}
\left|\tilde{\psi}(\bx)\right|\leq C\int_{0}^{|\bx|}\frac{1}{\langle r\rangle\langle\log\langle r\rangle\rangle}\,\rd r\leq C\log\langle\log\langle\bx\rangle\rangle\,.\label{eq:bound-psi}
\end{equation}
Now let $\tilde{\bv}_{n}=\bnabla^{\perp}\left(\eta_{n}\tilde{\psi}\right)$.
We have $\tilde{\bv}-\tilde{\bv}_{n}=\left(1-\eta_{n}\right)\tilde{\bv}-\tilde{\psi}\bnabla^{\perp}\eta_{n}$
so
\[
\bigl\Vert\bu\otimes\bigl(\tilde{\bv}-\tilde{\bv}_{n}\bigr)\bigr\Vert_{L^{2}(\Omega)}\leq\bigl\Vert\bigl(1-\eta_{n}\bigr)\bu\otimes\tilde{\bv}\bigr\Vert_{L^{2}(\Omega)}+\bigl\Vert\tilde{\psi}\bu\otimes\bnabla\eta_{n}\bigr\Vert_{L^{2}(\Omega)}\,.
\]
The first term goes to zero as $n\to\infty$ since $\bu\otimes\tilde{\bv}\in L^{2}(\Omega)$
because $\bu\weight\in L^{2}(\Omega)$ and $\tilde{\bv}/\weight\in L^{\infty}(\Omega)$.
Using the bound \eqref{eta-grad} on $\bnabla\eta_{n}$ and the bound
\eqref{bound-psi} on $\tilde{\psi}$, we obtain
\[
\bigl\Vert\tilde{\psi}\bu\otimes\bnabla\eta_{n}\bigr\Vert_{L^{2}(\Omega)}\leq\frac{C}{\log\langle\log\langle\log\langle n\rangle\rangle\rangle}\bigl\Vert\bu\weight\bigr\Vert_{L^{2}(\Omega)}\,,
\]
so the second term also goes to zero as $n\to\infty$, since $\bu\weight\in L^{2}(\Omega)$
in view of \lemref{generalized-hardy}. Finally, we have
\[
\bigl\Vert\bnabla\tilde{\bv}-\bnabla\tilde{\bv}_{n}\bigr\Vert_{L^{2}(\Omega)}\leq\bigl\Vert\bigl(1-\eta_{n}\bigr)\bnabla\tilde{\bv}\bigr\Vert_{L^{2}(\Omega)}+2\bigl\Vert\bnabla\eta_{n}\otimes\tilde{\bv}\bigr\Vert_{L^{2}(\Omega)}+\bigl\Vert\tilde{\psi}\bnabla^{2}\eta_{n}\bigr\Vert_{L^{2}(\Omega)}\,.
\]
The first term goes to zero since $\bnabla\tilde{\bv}\in L^{2}(\Omega)$.
For the second term, using \eqref{eta-grad} we have 
\[
\bigl\Vert\bnabla\eta_{n}\otimes\tilde{\bv}\bigr\Vert_{L^{2}(\Omega)}\leq\frac{C}{\log\langle\log\langle\log\langle n\rangle\rangle\rangle}\bigl\Vert\bv\weight\bigr\Vert_{L^{2}(\Omega)}\,,
\]
and using \eqref{eta-grad2} for the third term,
\[
\bigl\Vert\tilde{\psi}\bnabla^{2}\eta_{n}\bigr\Vert_{L^{2}(\Omega)}\leq\frac{C}{\log\langle\log\langle\log\langle n\rangle\rangle\rangle}\bigl\Vert\langle\bx\rangle^{-2}\bigr\Vert_{L^{2}(\Omega)}\,,
\]
so both converge to zero and $\tilde{\bv}_{n}\to\tilde{\bv}$ in $\dot{H}_{\sigma}^{1}(\Omega)$.
Finally, the sequence $\left(\tilde{\bv}_{n}\right)_{n\in\mathbb{N}}$
can be smoothed by using the standard mollification technique.
\end{proof}
Using the previous lemma, we can replace $\bphi$ by $\tilde{\bv}$
in the definition of the weak solution $\bu$:
\begin{lem}
\label{lem:ubar-in-u}If $\bu$ is a weak solution in $\Omega=\mathbb{R}^{2}$,
then
\[
\bigl\langle\bnabla\bu,\bnabla\tilde{\bv}\bigr\rangle_{L^{2}(\Omega)}+\bigl\langle\bu\bcdot\bnabla\bu,\tilde{\bv}\bigr\rangle_{L^{2}(\Omega)}=\bigl\langle\bbF,\bnabla\tilde{\bv}\bigr\rangle_{L^{2}(\Omega)}\,,
\]
for any $\tilde{\bv}\in\dot{H}_{\sigma}^{1}(\Omega)$ satisfying $\tilde{\bv}/\weight\in L^{\infty}(\Omega)$.\end{lem}
\begin{proof}
Let $\left(\tilde{\bv}_{n}\right)_{n\in\mathbb{N}}\subset C_{0,\sigma}^{\infty}(\Omega)$
be the approximation of $\tilde{\bv}$ constructed in \lemref{ubar-approx}.
Since $\bu$ is a weak solution, we have
\begin{equation}
\bigl\langle\bnabla\bu,\bnabla\tilde{\bv}_{n}\bigr\rangle_{L^{2}(\Omega)}+\bigl\langle\bu\bcdot\bnabla\bu,\tilde{\bv}_{n}\bigr\rangle_{L^{2}(\Omega)}=\bigl\langle\bbF,\bnabla\tilde{\bv}_{n}\bigr\rangle_{L^{2}(\Omega)}\,.\label{eq:ubarn-in-u}
\end{equation}
Since
\[
\left|\bigl\langle\bu\bcdot\bnabla\bu,\tilde{\bv}-\tilde{\bv}_{n}\bigr\rangle_{L^{2}(\Omega)}\right|\leq\bigl\Vert\bnabla\bu\bigr\Vert_{L^{2}(\Omega)}\bigl\Vert\bu\otimes\bigl(\tilde{\bv}-\tilde{\bv}_{n}\bigr)\bigr\Vert_{L^{2}(\Omega)}\,,
\]
by \lemref{ubar-approx}, we obtain the claimed result by passing
to the limit in \eqref{ubarn-in-u}.
\end{proof}
We can also replace $\bphi$ by $\bu$ in the definition of the weak
solution $\tilde{\bu}$:
\begin{lem}
\label{lem:u-in-ubar}If $\tilde{\bu}=\bu_{\infty}+\tilde{\bv}$ is
a weak solution in $\Omega=\mathbb{R}^{2}$ with $\bu_{\infty}\in\mathbb{R}^{2}$
and $\tilde{\bv}/\weight\in L^{\infty}(\Omega)$, then
\[
\bigl\langle\bnabla\tilde{\bv},\bnabla\bu\bigr\rangle_{L^{2}(\Omega)}-\bigl\langle\tilde{\bu}\bcdot\bnabla\bu,\tilde{\bv}\bigr\rangle_{L^{2}(\Omega)}=\bigl\langle\bbF,\bnabla\bu\bigr\rangle_{L^{2}(\Omega)}\,,
\]
for any $\bu\in\dot{H}_{\sigma}^{1}(\Omega)$.\end{lem}
\begin{proof}
By \propref{completion-omega}, let $\left(\bu_{n}\right)_{n\in\mathbb{N}}\subset C_{0,\sigma}^{\infty}(\Omega)$
be a sequence converging to $\bu$ in $\dot{H}_{\sigma}^{1}(\Omega)$.
Since $\tilde{\bu}=\bu_{\infty}+\tilde{\bv}$ is a weak solution,
we have
\[
\bigl\langle\bnabla\tilde{\bv},\bnabla\bu_{n}\bigr\rangle_{L^{2}(\Omega)}+\bigl\langle\tilde{\bu}\bcdot\bnabla\tilde{\bv},\bu_{n}\bigr\rangle_{L^{2}(\Omega)}=\bigl\langle\bbF,\bnabla\bu_{n}\bigr\rangle_{L^{2}(\Omega)}\,,
\]
or after an integration by parts,
\[
\bigl\langle\bnabla\tilde{\bv},\bnabla\bu_{n}\bigr\rangle_{L^{2}(\Omega)}-\bigl\langle\tilde{\bu}\bcdot\bnabla\bu_{n},\tilde{\bv}\bigr\rangle_{L^{2}(\Omega)}=\bigl\langle\bbF,\bnabla\bu_{n}\bigr\rangle_{L^{2}(\Omega)}\,.
\]
We can easily pass to the limit in the first and last terms. For the
second term, we have
\[
\left|\bigl\langle\tilde{\bu}\bcdot\bnabla\left(\bu-\bu_{n}\right),\tilde{\bv}\bigr\rangle_{L^{2}(\Omega)}\right|\leq\bigl\Vert\tilde{\bu}\otimes\tilde{\bv}\bigr\Vert_{L^{2}(\Omega)}\bigl\Vert\bnabla\bu-\bnabla\bu_{n}\bigr\Vert_{L^{2}(\Omega)}\leq C\bigl\Vert\bnabla\bu-\bnabla\bu_{n}\bigr\Vert_{L^{2}(\Omega)}\,,
\]
and the lemma is proven.
\end{proof}
We now prove the following consequence of the integration by parts:
\begin{lem}
\label{lem:ipp}For $\Omega=\mathbb{R}^{2}$, if $\tilde{\bv}\in\dot{H}_{\sigma}^{1}(\Omega)$
satisfies $\tilde{\bv}/\weight\in L^{\infty}(\Omega)$, then
\[
\bigl\langle\bu\bcdot\bnabla\tilde{\bv},\tilde{\bv}\bigr\rangle_{L^{2}(\Omega)}=0\,,
\]
for any $\bu\in\dot{H}_{\sigma}^{1}(\Omega)$.\end{lem}
\begin{proof}
Let $\left(\tilde{\bv}_{n}\right)_{n\in\mathbb{N}}\subset C_{0,\sigma}^{\infty}(\Omega)$
be the approximation of $\tilde{\bv}$ constructed in \lemref{ubar-approx}.
By integrating by parts, we have
\begin{equation}
\bigl\langle\bu\bcdot\bnabla\tilde{\bv},\tilde{\bv}_{n}\bigr\rangle_{L^{2}(\Omega)}+\bigl\langle\bu\bcdot\bnabla\tilde{\bv}_{n},\tilde{\bv}\bigr\rangle_{L^{2}(\Omega)}=0\,.\label{eq:ipp-approx}
\end{equation}
We have
\[
\left|\bigl\langle\bu\bcdot\bnabla\tilde{\bv},\bv-\tilde{\bv}_{n}\bigr\rangle_{L^{2}(\Omega)}\right|\leq\bigl\Vert\bnabla\tilde{\bv}\bigr\Vert_{L^{2}(\Omega)}\bigl\Vert\bu\otimes\bigl(\tilde{\bv}-\tilde{\bv}_{n}\bigr)\bigr\Vert_{L^{2}(\Omega)}\,,
\]
and
\[
\left|\bigl\langle\bu\bcdot\bnabla\bigl(\tilde{\bv}-\tilde{\bv}_{n}\bigr),\tilde{\bv}\bigr\rangle_{L^{2}(\Omega)}\right|\leq\bigl\Vert\bu\otimes\tilde{\bv}\bigr\Vert_{L^{2}(\Omega)}\bigl\Vert\bnabla\tilde{\bv}-\bnabla\tilde{\bv}_{n}\bigr\Vert_{L^{2}(\Omega)}\,,
\]
so by using \lemref{ubar-approx}, we can pass to the limit in \eqref{ipp-approx}
and the lemma is proven.
\end{proof}
We now can prove our weak-strong uniqueness results by using some
standard method (\citealp[Theorem X.3.2]{Galdi-IntroductiontoMathematical2011};
\citealt[Theorem 6]{Hillairet.Wittwer-Asymptoticdescriptionof2011}):
\begin{proof}[Proof of \thmref{uniqueness}]
Let $\tilde{\bv}=\tilde{\bu}-\bu_{\infty}$, $\bv=\bu-\bu_{\infty}$,
and $\bd=\bu-\tilde{\bu}=\bv-\tilde{\bv}$. By \lemref{ubar-in-u},
we have
\[
\bigl\langle\bnabla\bv,\bnabla\tilde{\bv}\bigr\rangle_{L^{2}(\Omega)}+\bigl\langle\bu\bcdot\bnabla\bv,\tilde{\bv}\bigr\rangle_{L^{2}(\Omega)}=\bigl\langle\bbF,\bnabla\tilde{\bu}\bigr\rangle_{L^{2}(\Omega)}\,,
\]
and by \lemref{u-in-ubar},
\[
\bigl\langle\bnabla\tilde{\bv},\bnabla\bv\bigr\rangle_{L^{2}(\Omega)}-\bigl\langle\tilde{\bu}\bcdot\bnabla\bv,\tilde{\bv}\bigr\rangle_{L^{2}(\Omega)}=\bigl\langle\bbF,\bnabla\bu\bigr\rangle_{L^{2}(\Omega)}\,,
\]
so, we obtain
\begin{eqnarray*}
\bigl\Vert\bnabla\bd\bigr\Vert_{L^{2}(\Omega)}^{2} & = & \bigl\Vert\bnabla\bu\bigr\Vert_{L^{2}(\Omega)}^{2}+\bigl\Vert\bnabla\tilde{\bu}\bigr\Vert_{L^{2}(\Omega)}^{2}-\bigl\langle\bnabla\bv,\bnabla\tilde{\bv}\bigr\rangle_{L^{2}(\Omega)}-\bigl\langle\bnabla\bv,\bnabla\tilde{\bv}\bigr\rangle_{L^{2}(\Omega)}\\
 & = & \bigl\Vert\bnabla\bu\bigr\Vert_{L^{2}(\Omega)}^{2}-\bigl\langle\bbF,\bnabla\tilde{\bu}\bigr\rangle_{L^{2}(\Omega)}+\bigl\Vert\bnabla\tilde{\bu}\bigr\Vert_{L^{2}(\Omega)}^{2}-\bigl\langle\bbF,\bnabla\bu\bigr\rangle_{L^{2}(\Omega)}+\bigl\langle\bd\bcdot\bnabla\bv,\tilde{\bv}\bigr\rangle_{L^{2}(\Omega)}\,.
\end{eqnarray*}
Using the energy inequality \eqref{energy-inequality} for both weak
solutions and \lemref{ipp},
\[
\bigl\Vert\bnabla\bd\bigr\Vert_{L^{2}(\Omega)}^{2}\leq\bigl\langle\bd\bcdot\bnabla\bv,\tilde{\bv}\bigr\rangle_{L^{2}(\Omega)}=\bigl\langle\bd\bcdot\bnabla\bd,\tilde{\bv}\bigr\rangle_{L^{2}(\Omega)}\leq\bigl\Vert\bnabla\bd\bigr\Vert_{L^{2}(\Omega)}\bigl\Vert\bd\tilde{\bv}\bigr\Vert_{L^{2}(\Omega)}^{2}\,,
\]
so by \lemref{generalized-hardy}, we obtain
\[
\bigl\Vert\bnabla\bd\bigr\Vert_{L^{2}(\Omega)}\leq\bigl\Vert\bd\tilde{\bv}\bigr\Vert_{L^{2}(\Omega)}^{2}\leq\bigl\Vert\bd\weight\bigr\Vert_{L^{2}(\Omega)}\bigl\Vert\tilde{\bv}/\weight\bigr\Vert_{L^{\infty}(\Omega)}^{2}\leq C\delta\bigl\Vert\bnabla\bd\bigr\Vert_{L^{2}(\Omega)}\,,
\]
since by hypothesis $\fint_{\omega}\bd=\bzero$. Therefore, for $\delta<C^{-1}$,
$\bnabla\bd=\bzero$, \emph{i.e.} $\bd=\bzero$.
\end{proof}

\subsubsection*{Acknowledgments}

The authors would like to thank M. Hillairet and V. Šverák for valuable
comments and suggestions on a preliminary version of the manuscript.
This research was partially supported by the Swiss National Science
Foundation grants \href{http://p3.snf.ch/Project-161996}{161996}
and \href{http://p3.snf.ch/Project-171500}{171500}.

\bibliographystyle{merlin-doi}
\bibliography{paper}

\begin{thebibliography}{21}
\providecommand{\natexlab}[1]{#1}
\expandafter\ifx\csname urlstyle\endcsname\relax
  \providecommand{\doi}[1]{doi:\discretionary{}{}{}#1}\else
  \providecommand{\doi}{doi:\discretionary{}{}{}\begingroup
  \urlstyle{rm}\Url}\fi

\bibitem[{Amick(1988)}]{Amick-Leraysproblemsteady1988}
\textsc{Amick, C.~J.} 1988, On {L}eray's problem of steady {N}avier--{S}tokes
  flow past a body in the plane. \emph{Acta Mathematica} \textbf{161}, 71--130,
  \doi{10.1007/BF02392295}

\bibitem[{Babenko(1970)}]{Babenko-AsymptoticBehaviorof1970}
\textsc{Babenko, K.~I.} 1970, The asymptotic behavior of a vortex far away from
  a body in a plane flow of viscous fluid. \emph{Journal of Applied Mathematics
  and Mechanics-USSR} \textbf{34}, 869--881, \doi{10.1016/0021-8928(70)90069-9}

\bibitem[{Deny \& Lions(1954)}]{Deny-Lesespaces1954}
\textsc{Deny, J. \& Lions, J.-L.} 1954, Les espaces du type de {B}eppo {L}evi.
  \emph{Annales de l'institut Fourier} \textbf{5}, 305--370,
  \doi{10.5802/aif.55}

\bibitem[{Finn \& Smith(1967)}]{Finn-stationarysolutionsNavier1967}
\textsc{Finn, R. \& Smith, D.~R.} 1967, On the stationary solutions of the
  {N}avier--{S}tokes equations in two dimensions. \emph{Archive for Rational
  Mechanics and Analysis} \textbf{25}, 26--39, \doi{10.1007/BF00281420}

\bibitem[{Galdi(2011)}]{Galdi-IntroductiontoMathematical2011}
\textsc{Galdi, G.~P.} 2011, \emph{An Introduction to the Mathematical Theory of
  the {N}avier--{S}tokes Equations. Steady-State Problems}. Springer Monographs
  in Mathematics, 2nd edn., Springer Verlag, New York,
  \doi{10.1007/978-0-387-09620-9}

\bibitem[{Gilbarg \&
  Trudinger(1998)}]{Gilbarg.Trudinger-EllipticPartialDifferential1988}
\textsc{Gilbarg, D. \& Trudinger, N.~S.} 1998, \emph{Elliptic Partial
  Differential Equations of Second Order}. Springer,
  \doi{10.1007/978-3-642-61798-0}

\bibitem[{Gilbarg \&
  Weinberger(1974)}]{Gilbarg.Weinberger-AsymptoticPropertiesof1974}
\textsc{Gilbarg, D. \& Weinberger, H.~F.} 1974, Asymptotic properties of
  {L}eray's solution of the stationary two-dimensional {N}avier--{S}tokes
  equations. \emph{Russian Mathematical Surveys} \textbf{29}~(2), 109--123,
  \doi{10.1070/RM1974v029n02ABEH003843}

\bibitem[{Gilbarg \&
  Weinberger(1978)}]{Gilbarg.Weinberger-Asymptoticpropertiesof1978}
\textsc{Gilbarg, D. \& Weinberger, H.~F.} 1978, Asymptotic properties of steady
  plane solutions of the {N}avier--{S}tokes equations with bounded {D}irichlet
  integral. \emph{Annali della Scuola Normale Superiore di Pisa}
  \textbf{5}~(2), 381--404

\bibitem[{Guillod(2015)}]{Guillod-review2015}
\textsc{Guillod, J.} 2015, Steady solutions of the {N}avier--{S}tokes equations
  in the plane, \arxiv{1511.03938}

\bibitem[{Hillairet \&
  Wittwer(2012)}]{Hillairet.Wittwer-Asymptoticdescriptionof2011}
\textsc{Hillairet, M. \& Wittwer, P.} 2012, Asymptotic description of solutions
  of the exterior {N}avier--{S}tokes problem in a half space. \emph{Archive for
  Rational Mechanics and Analysis} \textbf{205}, 553--584,
  \doi{10.1007/s00205-012-0515-6}

\bibitem[{Hillairet \& Wittwer(2013)}]{Hillairet-mu2013}
\textsc{Hillairet, M. \& Wittwer, P.} 2013, On the existence of solutions to
  the planar exterior {N}avier {S}tokes system. \emph{Journal of Differential
  Equations} \textbf{255}~(10), 2996--3019, \doi{10.1016/j.jde.2013.07.003}

\bibitem[{Korolev \&
  \v{S}ver\'ak(2011)}]{Korolev.Sverak-largedistanceasymptotics2011}
\textsc{Korolev, A. \& \v{S}ver\'ak, V.} 2011, On the large-distance
  asymptotics of steady state solutions of the {N}avier--{S}tokes equations in
  {3D} exterior domains. \emph{Annales de l'Institut Henri Poincar{\'{e}} -
  Analyse non lin{\'{e}}aire} \textbf{28}~(2), 303--313,
  \doi{10.1016/j.anihpc.2011.01.003}

\bibitem[{Kozono \& Sohr(1993)}]{Kozono-unbounded1993}
\textsc{Kozono, H. \& Sohr, H.} 1993, On stationary {N}avier--{S}tokes
  equations in unbounded domains. \emph{Ricerche di matematica} \textbf{42},
  69--86

\bibitem[{Leray(1933)}]{Leray-Etudedediverses1933}
\textsc{Leray, J.} 1933, {\'E}tude de diverses {\'e}quations int{\'e}grales non
  lin{\'e}aires et de quelques probl{\`e}mes que pose l'hydrodynamique.
  \emph{Journal de Math\'ematiques Pures et Appliqu\'ees} \textbf{12}, 1--82

\bibitem[{Nakatsuka(2015)}]{Nakatsuka-uniquenessofsymmetric2013}
\textsc{Nakatsuka, T.} 2015, On uniqueness of symmetric {N}avier--{S}tokes
  flows around a body in the plane. \emph{Advances in Differential Equations}
  \textbf{20}~(3/4), 193--212

\bibitem[{Ne\v{c}as(2012)}]{Necas-DirectMethods2012}
\textsc{Ne\v{c}as, J.} 2012, \emph{Direct Methods in the Theory of Elliptic
  Equations}. Springer Monographs in Mathematics, Springer Verlag,
  \doi{10.1007/978-3-642-10455-8}

\bibitem[{Pileckas \& Russo(2012)}]{Pileckas-existencevanishing2012}
\textsc{Pileckas, K. \& Russo, R.} 2012, On the existence of vanishing at
  infinity symmetric solutions to the plane stationary exterior
  {N}avier--{S}tokes problem. \emph{Mathematische Annalen} \textbf{352}~(3),
  643--658, \doi{10.1007/s00208-011-0653-4}

\bibitem[{Russo(2009)}]{Russo-NoteExteriorTwo-Dimensional2009}
\textsc{Russo, A.} 2009, A note on the exterior two-dimensional steady-state
  {N}avier--{S}tokes problem. \emph{Journal of Mathematical Fluid Mechanics}
  \textbf{11}~(3), 407--414, \doi{10.1007/s00021-007-0264-8}

\bibitem[{Sohr(2001)}]{Sohr-Navier-Stokesequations.elementary2001}
\textsc{Sohr, H.} 2001, \emph{The {N}avier--{S}tokes equations. An elementary
  functional analytic approach}. Birkh\"auser Advanced Texts, Birkh\"auser,
  Basel, \doi{10.1007/978-3-0348-8255-2}

\bibitem[{Yamazaki(2009)}]{Yamazaki-stationaryNavier-Stokesequation2009}
\textsc{Yamazaki, M.} 2009, The stationary {N}avier--{S}tokes equation on the
  whole plane with external force with antisymmetry. \emph{Annali
  dell'Universita di Ferrara} \textbf{55}, 407--423,
  \doi{10.1007/s11565-009-0080-6}

\bibitem[{Yamazaki(2011)}]{Yamazaki-Uniqueexistence2011}
\textsc{Yamazaki, M.} 2011, Unique existence of stationary solutions to the
  two-dimensional {N}avier--{S}tokes equations on exterior domains. In
  \emph{Mathematical Analysis on the {N}avier--{S}tokes equations and related
  topics, past and future} (edited by T.~Adachi, Y.~Giga, T.~Hishida, Y.~Kahei,
  H.~Kozono, \& T.~Ogawa), vol.~35 of \emph{GAKUTO International Series,
  Mathematical Sciences and Applications}, Gakkotosho

\end{thebibliography}

\end{document}